\documentclass[12pt]{amsart}

\usepackage{amsfonts}
\usepackage{amssymb}
\usepackage{amsthm}
\usepackage{amsmath}
\usepackage{enumerate}
\usepackage{mathtools}
\usepackage{graphicx}
\usepackage{color}

\newcommand{\eps}{\varepsilon}

\newcommand{\be}{\begin{equation}}
\newcommand{\ba}{\begin{aligned}}
\newcommand{\bee}{\begin{equation*}}
\newcommand{\ee}{\end{equation}}
\newcommand{\ea}{\end{aligned}}
\newcommand{\eee}{\end{equation*}}
\newcommand{\bea}{\begin{equation} \begin{aligned} }
\newcommand{\eea}{\end{aligned}\end{equation} }

\theoremstyle{plain}
\newtheorem{theorem}{Theorem}[section]

\newtheorem{conjecture}[theorem]{Conjecture} 
 
\newtheorem{claim}{Claim}[section]

\theoremstyle{remark}

\theoremstyle{definition}
\newtheorem{definition}[theorem]{Definition}

\numberwithin{equation}{section}

\begin{document}
\title[On $\kappa$-solutions and canonical neighborhoods]{On $\kappa$-solutions and canonical neighborhoods in 4d Ricci flow}
\author{Robert Haslhofer}

\begin{abstract}
We introduce a classification conjecture for $\kappa$-solutions in 4d Ricci flow. Our conjectured list
 includes known examples from the literature, but also a new 1-parameter family of $\mathbb{Z}_2^2\times \mathrm{O}_3$-symmetric bubble-sheet ovals that we construct. We observe that some special cases of the conjecture follow from recent results in the literature. We also introduce a stronger variant of the classification conjecture for ancient asymptotically cylindrical 4d Ricci flows, which does not assume smoothness and nonnegative curvature operator a priori. Assuming  this stronger variant holds true, we establish a canonical neighborhood theorem for 4d Ricci flow through cylindrical singularities, which shares some elements in common with Perelman's canonical neighborhood theorem for 3d Ricci flow as well as the mean-convex neighborhood theorem for mean curvature flow through neck-singularities. Finally, we argue that quotient-necks lead to new phenomena, and sketch an example of non-uniqueness for 4d Ricci flow through singularities.
\end{abstract}
\maketitle

\section{Introduction}

Recent groundbreaking work by Bamler \cite{Bam1,Bam2,Bam3} sparks hope towards constructing a Ricci flow through singularities in dimension 4, provided one gets a grasp on singularities.\\

In this paper, we introduce a conjectural picture of singularities in 4d Ricci flow, and derive some consequences.
To begin with, let us recall the notion of $\kappa$-solutions as introduced by Perelman \cite{Per1}:

\begin{definition}[$\kappa$-solution]
A \emph{$\kappa$-solution} is an ancient solution of the Ricci flow, complete with bounded curvature on compact time intervals, that has nonnegative curvature operator, positive scalar curvature, and is $\kappa$-noncollapsed at all scales.
\end{definition}

The relevance of $\kappa$-solutions is that they arise as blowup limits near cylindrical singularities. Simple examples to keep in mind are the non-degenerate neck-pinch and hole-punch singularity. For the neck-pinch, flowing into the singularity is modelled by a round shrinking $\mathbb{R}\times S^{n-1}$ and flowing out of the singularity is modelled by the Bryant soliton, see \cite{ACK}. For the hole-punch, flowing into the singularity is modelled by a round shrinking $\mathbb{R}^2\times S^{n-2}$ and flowing out of the singularity is modelled by an $(n-1)$-dimensional Bryant soliton times a line, see \cite{Carson}.\\

A qualitative description of $\kappa$-solutions in 3d Ricci flow played a key role in Perelman's proof of the geometrization conjecture \cite{Per2}.
In fact, by the recent classification of Brendle \cite{Brendle_ancient_Ricci} and Brendle-Daskalopoulos-Sesum \cite{BDS} any $\kappa$-solution in 3d Ricci flow is, up to scaling and finite quotients, either the round shrinking $S^3$, the round shrinking $\mathbb{R}\times S^2$, the 3d Bryant soliton, or the 3d Perelman oval. In particular, the classification of $\kappa$-solutions in 3d Ricci flow was an important ingredient in the recent proof of the generalized Smale conjecture by Bamler-Kleiner \cite{BK_diffeo,BK_contractibility,BK_prime}.\\

Here, we are concerned with $\kappa$-solutions in 4d Ricci flow. Up to scaling and finite quotients the known examples can be grouped into the following shrinkers, steadies and ovals. The shrinkers are $S^4$ with the round metric, $\mathbb{C}P^2$ with the Fubini-Study metric, $S^2\times S^2$ with the standard metric, the round shrinking neck $\mathbb{R}\times S^3$, and the round shrinking bubble-sheet $\mathbb{R}^2\times S^2$. The steadies are the $\mathrm{O}_4$-symmetric 4d steady soliton constructed by Bryant, the 3d Bryant-soliton times a line, and the 1-parameter family of $\mathbb{Z}_2\times\mathrm{O}_3$-symmetric steady solitons constructed by Lai \cite{Lai}. The known examples of ovals are the $\mathbb{Z}_2\times\mathrm{O}_4$-symmetric 4d ovals constructed by Perelman \cite{Per1}, the 3d ovals times a line, and the $\mathrm{O}_2\times\mathrm{O}_3$-symmetric 4d ovals constructed by Buttsworth \cite{Butt}.\\

In this paper, we first construct a new 1-parameter family of $\kappa$-solutions in 4d Ricci flow:

\begin{theorem}[bubble-sheet ovals]\label{thm_ovals}
On $S^4$, there exists a 1-parameter family of $\mathbb{Z}_2^2\times\mathrm{O}_3$-symmetric $\kappa$-solutions, whose tangent flow at $-\infty$ is a round shrinking $\mathbb{R}^2\times S^2$.
\end{theorem}

Loosely speaking, our 1-parameter family interpolates between $\mathbb{R}\times$3d-ovals and 3d-ovals$\times \mathbb{R}$.
Our construction is inspired by related examples for the mean curvature flow that we constructed in our recent work with Du \cite{DH_ovals}, see also \cite{HIMW}. Now, motivated by our classification of ancient noncollapsed mean curvature flows in $\mathbb{R}^4$ from joint work with B. Choi, K. Choi, Daskalopoulos, Du, Hershkovits and Sesum  \cite{CHH_wing,CHH_translator,DH_shape,DH_no_rotation,CDDHS,CHH_ancient} we conjecture that our new examples in fact complete the list:

\begin{conjecture}[$\kappa$-solutions in 4d Ricci flow]\label{conj_kappa}
Any $\kappa$-solution in 4d Ricci flow is, up to scaling and finite quotients, given by one of the following solutions.

\begin{itemize}
\item  shrinkers: $S^4$, $\mathbb{C}P^2$, $S^2\times S^2$, $\mathbb{R}\times S^3$ or $\mathbb{R}^2\times S^2$.
\item steadies: 4d Bryant soliton, the 3d Bryant soliton times a line, or belongs to the 1-parameter family of $\mathbb{Z}_2\times\mathrm{O}_3$-symmetric steady solitons constructed by Lai \cite{Lai}.
\item ovals: the $\mathbb{Z}_2\times\mathrm{O}_4$-symmetric 4d ovals from Perelman \cite{Per1}, the 3d ovals times a line, the $\mathrm{O}_2\times\mathrm{O}_3$-symmetric 4d ovals from Buttsworth \cite{Butt}, or belongs to the 1-parameter family of $\mathbb{Z}_2^2\times\mathrm{O}_3$-symmetric ovals constructed in Theorem \ref{thm_ovals}.
\end{itemize}
\end{conjecture}

We note that Conjecture \ref{conj_kappa} is true in the type I case by a recent result of Li \cite{Li_KR} (see also Lynch-RoyoAbrego \cite{LRA}). Moreover, by recent results Brendle-Naff \cite{BN} and Brendle-Daskalopoulos-Naff-Sesum \cite{BDNS} the conjecture is also true in case the tangent flow at $-\infty$ is a round shrinking neck.
Hence, the main remaining open problem is to deal with the bubble-sheet case. Motivated by \cite{CHH_wing}, a first key step in this case would be to show that for any noncompact $\kappa$-solution in 4d Ricci flow, whose tangent flow at  $-\infty$ is a round shrinking $\mathbb{R}^2\times S^2$, the Tits cone of any time-slice is either a ray or splits off a line.\\

We now turn to the related problem of finding canonical neighborhoods for 4d Ricci flow through singularities. For our purpose it is most convenient to describe Ricci flow through singularities using the notion of metric flows introduced by Bamler \cite{Bam2}. We recall that a metric flow
is given by a set $\mathcal{X}$, a time-function $\mathfrak{t}:\mathcal{X}\to \mathbb{R}$, complete separable metrics $d_t$ on the time-slices $\mathcal{X}_t=\mathfrak{t}^{-1}(t)$, and adjoint heat kernel measures $\nu_{x;s}\in \mathcal{P}(\mathcal{X}_s)$ such that the Kolmogorov consistency condition and Bamler's sharp gradient estimate for the heat flow hold. Specifically, we work with the following class of metric flows:

\begin{definition}[metric Ricci flow]
A \emph{4d metric Ricci flow} is a $4$-dimensional, noncollapsed, $(3\pi^2/2+4)$-concentrated, future-continuous metric flow (see \cite[Definition 3.30 and 4.25]{Bam2}) that satisfies the partial regularity properties from \cite{Bam3} and satisfies the Ricci flow equation with scalar curvature bounded below  on the regular part.
\end{definition}

For example, any noncollapsed limit of smooth Ricci flows, as provided by Bamler's precompactness theorem \cite{Bam2}, is a metric Ricci flow. We also remark that by our recent work with B. Choi \cite{ChoiHaslhofer} any metric Ricci flow is a weak solution in the sense of our joint work with Naber \cite{HaslhoferNaber}.\\

Now, given any 4d metric Ricci flow $\mathcal{X}$, any space-time point $x\in\mathcal{X}$, and any sequence $\lambda_i\to \infty$, we consider the sequence of flows $\mathcal{X}^{x,\lambda_i}$ that is obtained from $\mathcal{X}$ by parabolically rescaling by $\lambda_i$ around the center $x$, equipped with the parabolically rescaled adjoint heat kernel measures
$\nu^{x;\lambda_i}_s = \nu_{x;\mathfrak{t}(x)+\lambda_i^{-2}s}$. By Bamler's theory \cite{Bam2,Bam3}, the sequence of metric flow pairs $(\mathcal{X}^{x,\lambda_{i}},(\nu^{x;\lambda_{i}}_s)_{s\leq 0})$ always subsequentially converges to a limit $(\hat{\mathcal{X}},(\hat{\nu}_{\hat{x};s})_{s\leq 0})$, 
called a \emph{tangent flow at x}. Any such tangent flow is selfsimilar (either shrinking or static) and, since we are working in dimension 4, smooth except possibly for orbifold singularities. In fact, work of Munteanu-Wang \cite{MW1,MW2} suggests that tangent flows in 4d Ricci flow are either cylindrical or asymptotically conical. We are interested in finding a canonical description of the flow near any singular space-time point $x\in\mathcal{X}$. Since singularities with an asymptotically conical tangent flow are expected to be isolated, we can  focus on cylindrical singularities:

\begin{definition}[cylindrical singularity]\label{def_cyl_sing} We say that the flow has a cylindrical singularity at $x\in\mathcal{X}$ if some tangent flow at $x$ is either a round shrinking $\mathbb{R}\times S^3$ or a round shrinking $\mathbb{R}^2\times S^2$.\footnote{We remark that we do not allow for finite quotients.  As in Hamilton's paper on positive isotropic curvature \cite{Hamilton_pic}, one can exclude quotient-necks by imposing the topological assumption that there are no essential spherical space forms.}
\end{definition}

As we will see, finding a canonical neighborhood description of cylindrical singularities is closely related to the classification problem for ancient asymptotically cylindrical flows. To discuss this, recall from Bamler's work that if $\mathcal{X}$ is ancient, then similarly as above, but now with $\lambda_i\to 0$, the metric flow pairs $(\mathcal{X}^{x,\lambda_{i}},(\nu^{x;\lambda_{i}}_s)_{s\leq 0})$ always 
subsequentially converges to a limit$(\check{\mathcal{X}},(\check{\nu}_{\check{x};s})_{s\leq 0})$, called a \emph{tangent flow at $-\infty$}.

\begin{definition}[ancient asymptotically cylindrical 4d Ricci flow]
An \emph{ancient asymptotically cylindrical 4d Ricci flow} is an ancient 4d metric Ricci flow such that some tangent flow at $-\infty$ is either a round shrinking $\mathbb{R}\times S^3$ or a round shrinking $\mathbb{R}^2\times S^2$.
\end{definition}

Generalizing Conjecture \ref{conj_kappa} ($\kappa$-solutions in 4d Ricci flow) we propose:

\begin{conjecture}[ancient 4d Ricci flows]\label{conj_cyl}
Any ancient asymptotically cylindrical 4d Ricci flow is, up to scaling, either $\mathbb{R}\times S^3$ or $\mathbb{R}^2\times S^2$ or one of the steadies or ovals listed in Conjecture \ref{conj_kappa}.
\end{conjecture}

Conjecture \ref{conj_cyl} (ancient 4d Ricci flows) is motivated by our recent classification of ancient asymptotically cylindrical mean curvature flows from joint work with Choi, Hershkovits, and White \cite{CHH,CHHW}.\\

Assuming the conjecture, we establish the existence of canonical neighborhoods:

\begin{theorem}[canonical neighborhoods]\label{thm_can_nbd} Let $\mathcal{X}$ be a 4d metric Ricci flow that has a cylindrical singularity at $x\in\mathcal{X}$. Then, assuming Conjecture \ref{conj_cyl} (ancient 4d Ricci flows), for every $\eps>0$ there exists a $\delta=\delta(\eps,x)>0$, such that at any regular $y\in P(x,\delta)$ the scalar curvature satisfies $R(y)>\eps^{-1}$ and the flow that is obtained from $\mathcal{X}$ by centering at $y$ and parabolically rescaling by $R(y)^{1/2}$ is $\eps$-close in $C^{\lfloor 1/\eps\rfloor}$ in $P_{-}(0,1/\eps)$ to one of the solutions from above or to $S^4$.
\end{theorem}

Here, as in \cite[Definition 3.38]{Bam2}, we work with the parabolic neighborhoods $P(x,r)$ consisting of all space-time points $x'\in\mathcal{X}$ that satisfy $|\mathfrak{t}(x')-\mathfrak{t}(x)|\leq r^2$ and $d_{W_1}(\nu_{x;\mathfrak{t}(x)-r^2},\nu_{x';\mathfrak{t}(x)-r^2})\leq r$, where $d_{W_1}$ denotes the $1$-Wasserstein distance between probability measures in $\mathcal{X}_{\mathfrak{t}(x)-r^2}$. Likewise, $P_{-}(x,r)$ denotes the set of all $x'\in P(x,r)$ with $\mathfrak{t}(x')\leq\mathfrak{t}(x)$.\\

Theorem \ref{thm_can_nbd} (canonical neighborhoods) gives a precise description of the flow near any cylindrical singularity, and is inspired by Perelman's canonical neighborhood theorem for 3d Ricci flow \cite{Per2} and by the mean-convex neighborhood theorem for neck-singularities in mean curvature flow \cite{CHH,CHHW}.\\

Finally, let us say a few words regarding uniqueness of Ricci flow through singularities. In \cite{BK_uniqueness}, Bamler-Kleiner proved that 3d Ricci flow through singularities is unique. On the other hand, Angenent-Knopf \cite{AK_conical} showed that conical singularities can cause nonuniqueness in dimension 5 and higher, leaving open the critical dimension 4. Motivated by Theorem \ref{thm_can_nbd} (canonical neighborhoods) and the uniqueness result for mean curvature flow through singularities with mean-convex neighborhoods from Hershkovits-White \cite{HW_uniqueness}, it seems reasonable to expect that 4d Ricci flow through cylindrical singularities should be unique, see also \cite{Haslhofer_Ricci_note}. On the other hand, we believe that quotient-necks lead to non-uniqueness:

\begin{conjecture}[non-uniqueness]\label{conj_nonunique}
For every $k\geq 3$ there exist a Ricci flow on $S^1\times S^3/\mathbb{Z}_k$ that forms a quotient-neck singularity with tangent flow $\mathbb{R}\times S^3/\mathbb{Z}_k$, and continuous non-uniquely modelled either on the Bryant soliton modulo $\mathbb{Z}_k$ or on Appleton's cohomogeneity-one soliton on the line bundle $O(-k)$ from \cite{Appleton_soliton}.
\end{conjecture}

The conjecture proposes a mechanism for non-uniqueness in dimension 4.  To prove Conjecture \ref{conj_nonunique} (non-uniqueness), in light of work of Angenent-Caputo-Knopf \cite{ACK}, it would suffice to show that quotient-necks can be resolved by gluing in Appleton solitons.\\

This article is organized as follows. In Section \ref{sec2}, we prove Theorem \ref{thm_ovals} (bubble-sheet ovals). In Section \ref{sec3}, we prove Theorem \ref{thm_can_nbd} (canonical neighborhoods).

\bigskip
\noindent\textbf{Acknowledgments.} We thank the referee for useful comments. The author has been partially supported by an NSERC Discovery Grant.\\

\section{Construction of new $\kappa$-solutions}\label{sec2}

In this section, we prove Theorem \ref{thm_ovals} (bubble-sheet ovals).

\begin{proof}[Proof of Theorem \ref{thm_ovals}] 
Given any $a\in (0,1)$ and $L < \infty$, we consider the ellipsoid
\begin{equation}
E^{L,a}:=\left\{ x\in \mathbb{R}^{5} \, : \,  \frac{a^2}{L^2} x_1^2 + \frac{(1-a)^2}{L^2} x_2^2 + x_3^2+x_4^2+x_5^2 = 1 \right\}\, .
\end{equation}
Denote by $g^{L,a}$ the metric on $S^4$ induced by the canonical map $S^4\to E^{L,a}$.\\

By a result of Hamilton \cite{Hamilton_pco}, the Ricci flow evolution $g^{L,a}_t$ has positive curvature operator and becomes extinct in a round point at some $t_{L,a}<\infty$.  
Since the Ricci flow depends continuously on the initial condition, we have
\begin{equation}\label{eq_t_lim}
\lim_{L\to\infty} t_{L,a}=\frac{1}{2}.
\end{equation}
Also, from the explicit formula for ellipsoids it is not hard to see that metrics $g^{L,a}$ have curvature uniformly bounded above, and injectivity radius uniformly bounded below. Hence, by Perelman \cite{Per1}, there is some uniform $\kappa>0$ such that the metrics $g^{L,a}_t$ are $\kappa$-noncollapsed at scales $\leq 1$. Note also that the $\mathbb{Z}_2^2\times \mathrm{O}_3$-symmetry is preserved under Ricci flow, parabolic rescaling, and passing to limits.\\

Consider Perelman's monotone quantity \cite{Per1},
\begin{equation}
V^{L,a}(t)=\int_{S^4} \frac{1}{(4\pi(t_{L,a}-t))^{2}} e^{-\ell(q,t)}\, dV_{g^{L,a}_t}(q),
\end{equation}
where $\ell$ denotes the reduced length based at the singular time, c.f. Enders-M\"uller-Topping \cite{EMT}.
Then, for any large enough $L$ we can find a unique $t_{L,a}'$ such that
\begin{equation}
V^{L,a}(t_{L,a}')=\frac{v_{2}+v_{3}}{2},
\end{equation}
where $v_j$ denotes the reduced volume of the $j$-sphere (concretely, one has $v_2=2/e$ and $v_3=2(\pi/e^3)^{1/2}$, see e.g. \cite{CHI}).
Using again continuous dependence of the Ricci flow on the initial data we see that
\begin{equation}\label{eq_tp_lim}
\lim_{L\to\infty} t_{L,a}'=\frac{1}{2}.
\end{equation}
Now, set
\begin{equation}
\lambda_{L,a}:=(t_{L,a}-t_{L,a}')^{-1/2},
\end{equation}
and consider the parabolically rescaled flows
\begin{equation}
\tilde{g}^{L,a}_t := \lambda_{L,a}^2 g^{L,a}_{\lambda^{-2}_{L,a} t+t_{L,a}}.
\end{equation}
By construction $\tilde{g}^{L,a}_t$ becomes extinct at time $0$ and satisfies
\begin{equation}
\int_{S^4} \frac{1}{(4\pi)^{2}}e^{-\ell(q,-1)}\, dV_{\tilde{g}^{L,a}_{-1}}(q)=\frac{v_{2}+v_{3}}{2}.
\end{equation}
The flow $\tilde{g}^{L,a}_t$ is defined for $t\in (T_{L,a},0)$, where $T_{L,a}=-\lambda_{L,a}^2 t_{L,a}$, and thanks to \eqref{eq_t_lim} and \eqref{eq_tp_lim} we have
\begin{equation}
\lim_{L\to\infty} T_{L,a} = -\infty.
\end{equation}

We now consider suitable widths in $x_1$ and $x_2$ direction. Specifically, we set
\begin{equation}
    w_{1}^L(a):=\mathrm{diam}\big(\tilde{g}^{L,a}_{-1}|_{S^4\cap \{ x_2 = 0\}}\big),\qquad  w_{2}^L(a):=\mathrm{diam}\big(\tilde{g}^{L,a}_{-1}|_{S^4\cap \{ x_1 = 0\}}\big).
\end{equation}
Using this we can now define the reciprocal width ratio map
\begin{equation}
F^{L} : (0,1) \to (0,1), \quad a\mapsto \frac{w_{1}^L(a)^{-1}}{w_{1}^L(a)^{-1}+w_{2}^L(a)^{-1}}\, .
\end{equation}

\begin{claim}[reciprocal width ratio map]\label{claim_width}
$F^{L}$ is continuous and surjective.
\end{claim}

\begin{proof}
Let $a^{i}\in (0,1)$ be a sequence that converges to some $a\in (0,1)$. Then, clearly $g^{L,a_i}\to g^{L,a}$, and arguing similarly as above we also see that $t_{L,a^{i}}\to t_{L,a}$ and $\lambda_{L,a_i}\to \lambda_{L,a}$. Now, by Hamilton's compactness theorem \cite{Hamilton_compactness} there is a subsequence such that $(S^4,\tilde{g}^{L,a_i}_t)$ converges to some $\kappa$-noncollasped limit. By the above and by uniqueness of Ricci flow this limit must be equal to $(S^4,\tilde{g}^{L,a}_t)$. This shows that $F^{L}(a_i)\to F^{L}(a)$. Furthermore, observe that
\begin{equation}
\lim_{a\to 0} F^L(a)=0,\qquad \lim_{a\to 1} F^L(a)=1.
\end{equation}
Together with the intermediate value theorem this implies the assertion.
\end{proof}

Continuing the proof of the theorem, for any $\mu\in (0,1)$ we will now construct a bubble-sheet oval with prescribed reciprocal width ratio $\mu$. To this end, given
 $\mu_i\to \mu$ and $L_i\to \infty$, by Claim \ref{claim_width} (reciprocal width ratio) we can find $a_i \in (0,1)$, such that $F^{L_i}(a_i)=\mu_i$. 
Now, by \cite[Section 7]{Per1} our sequence of rescaled flows satisfies
\begin{equation}
\sup_{t<0} |t|\min_{x\in S^4}R_{\tilde{g}^{L_i,a_i}_t}(x) \leq C.
\end{equation}
Hence, by Hamilton's Harnack inequality and compactness theorem \cite{Hamilton_Harnack,Hamilton_compactness}, after passing to a subsequence $(S^4,\tilde{g}^{L_i,a_i}_t)$ converges to an ancient Ricci flow $(M,g_t)_{t<0}$ with nonnegative curvature operator that is $\kappa$-noncollapsed at all scales. Note that the scalar curvature is positive by the strict maximum principle. Moreover, by construction the limit flow becomes extinct at time $0$, is $\mathbb{Z}_2^2\times \mathrm{O}_3$-symmetric, and satisfies
\begin{equation}\label{withcond}
\int_{M} \frac{1}{(4\pi)^{2}}e^{-\ell(q,-1)}\, dV_{g_{-1}}(q)=\frac{v_{2}+v_{3}}{2}.
\end{equation}
We claim that the limit is compact. To see this, suppose towards a contradiction that $w_1^{L_i}(a_i)\to \infty$ and $w_2^{L_i}(a_i)\to \infty$. Then $(M,g_t)_{t<0}$ would split off two lines, and hence would be a round shrinking bubble-sheet, contradicting \eqref{withcond}. Remembering also $F^{L_i}(a_i)=\mu_i$, this shows that $w_1^{L_i}(a_i)+w_2^{L_i}(a_i)$ is bounded. Thus, $M$ is compact, and hence diffeomorphic to $S^4$. Also, by construction the limit satisfies
\begin{equation}
\frac{w_{1}^L(a)^{-1}}{w_{1}^L(a)^{-1}+w_{2}^L(a)^{-1}}=\mu.
\end{equation}
Finally, by \cite[Section 11]{Per1} and \cite[Corollary 4]{MW3}, any tangent flow at $-\infty$ must be a generalized cylinder, and together with \eqref{withcond} and the symmetries it follows that it must be a bubble-sheet. This concludes the proof of the theorem.
\end{proof}

\bigskip

\section{Canonical neighborhoods}\label{sec3}

In this section, we prove Theorem \ref{thm_can_nbd} (canonical neighborhoods).

\begin{proof}[Proof of Theorem \ref{thm_can_nbd}] 
Let $\mathcal{X}$ be a 4d metric Ricci flow that has a cylindrical singularity at $x\in\mathcal{X}$. By Definition \ref{def_cyl_sing} (cylindrical singularity) this means that there is some sequence $\lambda_i\to \infty$, such that the metric flow pair $(\mathcal{X}^{x,\lambda_{i}},(\nu^{x;\lambda_{i}}_s)_{s\leq 0})$, which is obtained from $\mathcal{X}$ by parabolically rescaling by $\lambda_i$ around the center $x$, converges to a round shrinking $\mathbb{R}\times S^3$ or a round shrinking $\mathbb{R}^2\times S^2$, equipped with the standard adjoint heat kernel measures. In particular, by the semicontinuity of the Nash entropy from \cite[Proposition 4.37]{Bam3}, which is applicable thanks to our assumption that the scalar curvature is bounded below, the flow $\mathcal{X}$ is $\kappa$-noncollapsed in the two-sided parabolic ball $P(x,\delta_0)$ for some constants $\kappa>0$ and $\delta_0>0$.
 A priori the above limit is just  in the sense of $\mathbb{F}$-convergence on compact time-intervals within some correspondence \cite[Section 6]{Bam2}. However, since the limit is smooth, by the local regularity theorem \cite[Theorem 2.29]{Bam3} (see also Hein-Naber \cite{HeinNaber}) the convergence is actually locally smooth. Moreover, since cylinders are isolated in the space of shrinkers by a result of Colding-Minicozzi \cite{CM_Ricci} (see also Li-Wang \cite{LW}), the flow pair $(\mathcal{X}^{x,\lambda_i},(\nu^{x;\lambda_i}_s)_{s\leq 0})$ actually converges to a round shrinking cylinder along every sequence $\lambda_i\to \infty$. \\

Now, suppose towards a contradiction there are regular points $y_i\to x\in \mathcal{X}$ whose scalar curvature satisfies $R(y_i)\leq C$.
Consider the regularity scale $r_{\textrm{reg}}(y_i)$, i.e. the largest radius $r\leq 1$ such that $\mathcal{X}$ is smooth with curvature bounded by $r^{-2}$ in the parabolic ball $P(y_i,r)$. Since $y_i$ is a sequence of regular points that converges to a singular point, the numbers $\lambda_i:=r_{\textrm{reg}}(y_i)^{-1}$ are finite and converge to infinity. Consider the sequence of rescaled flows $(\mathcal{X}^{y_i,\lambda_i},(\nu^{y_i;\lambda_i}_s)_{s\leq 0})$. By Bamler's compactness theorem \cite{Bam2}, after passing to a subsequence, we can assume that it converges to some limit $(\mathcal{X}^\infty,(\nu^\infty_{x^\infty;s})_{s\leq 0})$.\\

Since $y_i\to x$, by the first paragraph, taking also into account Bamler's change of base-point theorem from \cite[Section 6]{Bam2}, the limit flow $\mathcal{X}^\infty$ is either an ancient asymptotically cylindrical flow or a nontrivial blowup limit thereof.
Hence, assuming Conjecture \ref{conj_cyl} (ancient 4d Ricci flows), and observing that taking nontrivial blowups thereof only adds $S^4$ to the list, $\mathcal{X}^\infty$ must be either $S^4$, $\mathbb{R}\times S^3$, $\mathbb{R}^2\times S^2$ or one of the listed steadies or ovals. In particular, all these solutions have strictly positive scalar curvature. On the other hand, by construction $0\in \mathcal{X}^\infty$ has regularity scale at least 1, and thus in particular is a regular point. Hence, $R(0)\leq0$, which gives the desired contradiction.\\

So far we have shown that for every $\eps>0$ there exists a $\delta_1=\delta_1(\eps,x)>0$ such that at any regular $y\in P(x,\delta_1)$ the scalar curvature satisfies $R(y)>\eps^{-1}$. To proceed, recall that by definition the regularity scale is bounded above by the scalar curvature scale, namely $r_{\textrm{reg}}\leq R^{-1/2}$. On the other hand, if along some sequence of regular points $y_i\to x$ we had $r_{\textrm{reg}}(y_i) R(y_i)^{1/2}\to 0$, then blowing by $\lambda_i:=r_{\textrm{reg}}(y_i)^{-1}$ we would obtain a similar contradiction as above. This shows that these two scales are comparable near $x$, namely there exist some $\delta_2=\delta_2(x)>0$ and $C=C(x)<\infty$, such at any regular point $y\in P(x,\delta_2)$ we have $r_{\textrm{reg}}(y)\leq R^{-1/2}(y)\leq Cr_{\textrm{reg}}(y)$.\\

Finally, suppose towards a contradiction there are regular points $y_i\to x$, such that the rescaled flows $\mathcal{X}^{y_i,\lambda_i}$, where $\lambda_i:=R(y_i)^{1/2}$, are not $\eps$-close in $C^{\lfloor 1/\eps\rfloor}$ in the backwards parabolic ball $P_{-}(0,1/\eps)$ to one of the solutions from Conjecture \ref{conj_cyl} (ancient 4d Ricci flows) or to $S^4$. Then, arguing as above along a subsequence we could pass to a limit $\mathcal{X}^\infty$ that must be either $S^4$, $\mathbb{R}\times S^3$, $\mathbb{R}^2\times S^2$ or one of the listed steadies or ovals, and such that $0\in \mathcal{X}^\infty$ has regularity scale comparable to $1$. However, by the local regularity theorem \cite[Theorem 2.29]{Bam3} (see also \cite{HeinNaber}) this implies $\eps$-closeness for $i$ large enough. This gives the desired contradiction, and thus concludes the proof of the theorem.
\end{proof}

\bigskip

\bibliography{4d_ricci.bib}
\bibliographystyle{alpha}

\vspace{10mm}

{\sc Department of Mathematics, University of Toronto,  40 St George Street, Toronto, ON M5S 2E4, Canada}\\

\end{document}